\documentclass[12pt]{article}%
\usepackage{amssymb,latexsym}
\usepackage{amsfonts}
\usepackage{amsmath}
\usepackage[colorlinks=true, pdfstartview=FitV, linkcolor=blue,
citecolor=blue, urlcolor=blue]{hyperref}
\usepackage{color}
\usepackage{amssymb}
\usepackage{graphicx}
\usepackage[compress]{cite}%
\providecommand{\U}[1]{\protect\rule{.1in}{.1in}}
\newtheorem{theorem}{Theorem}

\newtheorem{lemma}[theorem]{Lemma}

\newtheorem{remark}[theorem]{Remark}

\newenvironment{proof}[1][Proof]{\noindent\textbf{#1.} }{\ \rule{0.5em}{0.5em}}
\numberwithin{equation}{section}
\numberwithin{theorem}{section}
\allowdisplaybreaks
\begin{document}

\title{Closed formulas and determinantal expressions for higher-order Bernoulli and
Euler polynomials in terms of Stirling numbers}
\author{Muhammet Cihat DA\u{G}LI\\Department of Mathematics, Akdeniz University,\\07058-Antalya, Turkey\\E-mail: mcihatdagli@akdeniz.edu.tr}
\date{}
\maketitle

\begin{abstract}
In this paper, applying the Fa\`{a} di Bruno formula and some properties of
Bell polynomials, several closed formulas and determinantal expressions
involving Stirling numbers of the second kind for higher-order Bernoulli and
Euler polynomials are presented.



\end{abstract}

\section{Introduction}

The classical Bernoulli polynomials $B_{n}(x)$ and Euler polynomials
$E_{n}(x)$ are usually defined by means of the following generating functions:%
\[
\dfrac{te^{xt}}{e^{t}-1}=%
{\displaystyle\sum\limits_{n=0}^{\infty}}
B_{n}(x)\dfrac{t^{n}}{n!}\text{ }\left(  \left\vert t\right\vert <2\pi\right)
\text{ and }\dfrac{2e^{xt}}{et+1}=%
{\displaystyle\sum\limits_{n=0}^{\infty}}
E_{n}(x)\dfrac{t^{n}}{n!}\text{ }\left(  \left\vert t\right\vert <\pi\right)
.
\]
In particular, the rational numbers $B_{n}=B_{n}(0)$ and integers $E_{n}%
=2^{n}E_{n}(1/2)$ are called classical Bernoulli numbers and Euler numbers, respectively.

As is well known, the classical Bernoulli and Euler polynomials play important
roles in different areas of mathematics such as number theory, combinatorics,
special functions and analysis.

Numerous generalizations of these polynomials and numbers are defined and many
properties are studied in a variety of context. One of them can be traced back
to N\"{o}rlund \cite{n}: The higher-order Bernoulli polynomials $B_{n}%
^{(\alpha)}(x)$ and higher-order Euler polynomials $E_{n}^{(\alpha)}(x)$, each
of degree $n$ in $x$ and in $\alpha$, are defined by means of the generating
functions
\begin{equation}
\left(  \dfrac{t}{e^{t}-1}\right)  ^{\alpha}e^{xt}=%
{\displaystyle\sum\limits_{n=0}^{\infty}}
B_{n}^{(\alpha)}(x)\dfrac{t^{n}}{n!} \label{0}%
\end{equation}
and%
\begin{equation}
\text{ }\left(  \dfrac{2}{e^{t}+1}\right)  ^{\alpha}e^{xt}=%
{\displaystyle\sum\limits_{n=0}^{\infty}}
E_{n}^{(\alpha)}(x)\dfrac{t^{n}}{n!}, \label{15}%
\end{equation}
respectively. For $\alpha=1$, we have $B_{n}^{(1)}(x)=B_{n}(x)$ and
$E_{n}^{(1)}(x)=E_{n}(x)$.

According to Wiki \cite{closed}, "In mathematics, a closed-form expression is
a mathematical expression that can be evaluated in a finite number of
operations. It may contain constants, variables, certain `well-known'
operations (e.g.,$+-\times\div$), and functions (e.g., $n$th root, exponent,
logarithm, trigonometric functions, and inverse hyperbolic functions), but
usually no limit."

From this point of view, Wei and Qi \cite{wei} studied several closed form
expressions for Euler polynomials in terms of determinant and the Stirling
numbers of the second kind. Also, Qi and Chapman \cite{qi2} established two
closed forms for the Bernoulli polynomials and numbers involving the Stirling
numbers of the second kind and in terms of a determinant of combinatorial
numbers. Moreover, some special determinantal expressions and recursive
formulas for Bernoulli polynomials and numbers can be found in \cite{qi6}.

In 2018, Hu and Kim \cite{hu-kim} presented two closed formulas for
Apostol-Bernoulli polynomials by aid of the properties of the Bell polynomials
of the second kind $B_{n,k}\left(  x_{1},x_{2},...,x_{n-k+1}\right)  $ (see
Lemma \ref{3}, below). Very recently, Dai and Pan \cite{d} have obtained the
closed forms for degenerate Bernoulli polynomials.

In this paper, we focus on higher-order Bernoulli and Euler polynomials in
those respects, mentioned above. Firstly, we find some novel closed formulas
for higher-order Bernoulli and Euler polynomials in terms of Stirling numbers
of the second kind $S(n,k)$ via the Fa\`{a} di Bruno formula for the Bell
polynomials of the second kind and the generating function methods. Secondly,
we establish some determinantal expressions by applying a formula of higher
order derivative for the ratio of two differentiable functions. Consequently,
taking some special cases in our results provides the further formulas for
classical Bernoulli and Euler polynomials, and numbers.

\section{Some Lemmas}

In order to prove our main results, we recall several lemmas below.

\begin{lemma}
\label{3}(\cite[p. 134 and 139]{Comtet}) The Bell polynomials of the second
kind, or say, partial Bell polynomials, denoted by $B_{n,k}\left(  x_{1}%
,x_{2},...,x_{n-k+1}\right)  $ for $n\geq k\geq0,$ defined by
\[
B_{n,k}\left(  x_{1},x_{2},...,x_{n-k+1}\right)  =%
{\displaystyle\sum\limits_{\substack{1\leq i\leq n,\text{ }l_{i}\in\left\{
0\right\}  \cup\mathbb{N}\\{{\textstyle\sum\nolimits_{i=1}^{n}}}%
il_{i}=n,\text{ }{{\textstyle\sum\nolimits_{i=1}^{n}}}l_{i}=k}}^{\infty}}
\frac{n!}{%
{\textstyle\prod\nolimits_{i=1}^{l-k+1}}
l_{i}!}%
{\textstyle\prod\limits_{i=1}^{l-k+1}}
\left(  \frac{x_{i}}{i!}\right)  ^{l_{i}}.
\]
The Fa\`{a} di Bruno formula can be described in terms of the Bell polynomials
of the second kind $B_{n,k}\left(  x_{1},x_{2},...,x_{n-k+1}\right)  $ by%
\begin{equation}
\frac{d^{n}}{dt^{n}}f\circ h\left(  t\right)  =%
{\displaystyle\sum\limits_{k=0}^{n}}
f^{\left(  k\right)  }\left(  h\left(  t\right)  \right)  B_{n,k}\left(
h^{\prime}\left(  t\right)  ,h^{\prime\prime}\left(  t\right)
,...,h^{(n-k+1)}\left(  t\right)  \right)  . \label{2}%
\end{equation}

\end{lemma}

\begin{lemma}
\label{4}(\cite[p. 135]{Comtet}) For $n\geq k\geq0,$ we have%
\begin{equation}
B_{n,k}\left(  abx_{1},ab^{2}x_{2},...,ab^{n-k+1}x_{n-k+1}\right)  =a^{k}%
b^{n}B_{n,k}\left(  x_{1},x_{2},...,x_{n-k+1}\right)  , \label{5}%
\end{equation}
where $a$ and $b$ are any complex number.
\end{lemma}

\begin{lemma}
(\cite[p. 135]{Comtet})For $n\geq k\geq0,$ we have%
\begin{equation}
B_{n,k}\left(  1,1,...,1\right)  =S\left(  n,k\right)  , \label{12}%
\end{equation}
where $S\left(  n,k\right)  $ is the Stirling numbers of the second kind,
defined by \cite[p. 206]{Comtet}
\[
\frac{\left(  e^{t}-1\right)  ^{k}}{k!}=%
{\displaystyle\sum\limits_{n=k}^{\infty}}
S\left(  n,k\right)  \frac{t^{n}}{n!}.
\]

\end{lemma}

\begin{lemma}
(\cite{guo})For $n\geq k\geq1,$ we have%
\begin{equation}
B_{n,k}\left(  \frac{1}{2},\frac{1}{3},...,\frac{1}{n-k+2}\right)  =\frac
{n!}{\left(  n+k\right)  !}%
{\displaystyle\sum\limits_{i=0}^{k}}
\left(  -1\right)  ^{k-i}\binom{n+k}{k-i}S\left(  n+i,i\right)  . \label{6}%
\end{equation}

\end{lemma}

\begin{lemma}
\label{1}(\cite[p. 40, Entry 5]{bourbaki}) For two differentiable functions
$p\left(  x\right)  $ and $q\left(  x\right)  \neq0,$ we have for $k\geq0$%
\begin{align}
&  \frac{d^{k}}{dx^{k}}\left[  \frac{p\left(  x\right)  }{q\left(  x\right)
}\right] \nonumber\\
&  =\frac{\left(  -1\right)  ^{k}}{q^{k+1}}\left\vert
\begin{array}
[c]{cccccccc}%
p & q & 0 & . & . & . & 0 & 0\\
p^{\prime} & q^{\prime} & q & . & . & . & 0 & 0\\
p^{\prime\prime} & q^{\prime\prime} & \binom{2}{1}q^{\prime} & . & . & . & 0 &
0\\
. & . & . & . & . & . & . & .\\
. & . & . & . & . & . & . & .\\
. & . & . & . & . & . & . & .\\
p^{\left(  k-2\right)  } & q^{\left(  k-2\right)  } & \binom{k-2}{1}q^{(k-3)}
& . & . & . & q & 0\\
p^{\left(  k-1\right)  } & q^{\left(  k-1\right)  } & \binom{k-1}{1}q^{(k-2)}
& . & . & . & \binom{k-1}{k-2}q^{\prime} & q\\
p^{\left(  k\right)  } & q^{\left(  k\right)  } & \binom{k}{1}q^{(k-1)} & . &
. & . & \binom{k}{k-2}q^{\prime\prime} & \binom{k}{k-1}q^{\prime}%
\end{array}
\right\vert \label{10}%
\end{align}
In other words, the formula (\ref{10}) can be represented as%
\[
\frac{d^{k}}{dx^{k}}\left[  \frac{p\left(  x\right)  }{q\left(  x\right)
}\right]  =\frac{\left(  -1\right)  ^{k}}{q^{k+1}}\left\vert W_{\left(
k+1\right)  \times\left(  k+1\right)  }\left(  x\right)  \right\vert ,
\]
where $\left\vert W_{\left(  k+1\right)  \times\left(  k+1\right)  }\left(
x\right)  \right\vert $ denotes the determinant of the matrix%
\[
W_{\left(  k+1\right)  \times\left(  k+1\right)  }\left(  x\right)  =\left[
U_{\left(  k+1\right)  \times1}\left(  x\right)  \text{ \ \ }V_{\left(
k+1\right)  \times k}\left(  x\right)  \right]  .
\]
Here $U_{\left(  k+1\right)  \times1}\left(  x\right)  $ has the elements
$u_{l,1}\left(  x\right)  =p^{\left(  l-1\right)  }\left(  x\right)  $ for
$1\leq l\leq k+1$ and $V_{\left(  k+1\right)  \times k}\left(  x\right)  $ has
the entries of the form%
\[
v_{i,j}\left(  x\right)  =%
\begin{cases}
\binom{i-1}{j-1}q^{\left(  i-j\right)  }\left(  x\right)  , & \text{if
}i-j\geq0\text{;}\\
0, & \text{if }i-j<0\text{,}%
\end{cases}
\]
for $1\leq i\leq k+1$ and $1\leq j\leq k.$
\end{lemma}

\section{Closed formulas}

In this section, we give closed formulas for higher-order Bernoulli and Euler polynomials.

\begin{theorem}
\label{main}The higher-order Bernoulli polynomials $B_{n}^{(\alpha)}(x)$ can
be expressed as
\[
B_{n}^{(\alpha)}(x)=%
{\displaystyle\sum\limits_{k=0}^{n}}
\binom{n}{k}%
{\displaystyle\sum\limits_{i=0}^{k}}
\left\langle -\alpha\right\rangle _{i}\frac{k!}{(k+i)!}%
{\displaystyle\sum\limits_{j=0}^{i}}
\left(  -1\right)  ^{i-j}\binom{k+i}{i-j}S(k+j,j)x^{n-k},
\]
where $S(n,k)$ is the Stirling numbers of the second kind and $\left\langle
x\right\rangle _{n}$ denotes the falling factorial, defined for $x\in%
\mathbb{R}
$ by
\[
\left\langle x\right\rangle _{n}=\prod\limits_{k=0}^{n-1}\left(  x-k\right)  =%
\begin{cases}
x\left(  x-1\right)  ...(x-n+1), & \text{if }n\geq1\text{;}\\
1, & \text{if }n=0.
\end{cases}
\]
In particular $x=0,$ the higher-order Bernoulli numbers $B_{n}^{(\alpha)}$
possess the following form%
\begin{equation}
B_{n}^{(\alpha)}=%
{\displaystyle\sum\limits_{i=0}^{n}}
\left\langle -\alpha\right\rangle _{i}\frac{n!}{(n+i)!}%
{\displaystyle\sum\limits_{j=0}^{i}}
\left(  -1\right)  ^{i-j}\binom{n+i}{i-j}S(n+j,j). \label{11}%
\end{equation}

\end{theorem}

\begin{proof}
Let us begin by writing $\left(  \dfrac{e^{t}-1}{t}\right)  ^{\alpha}=\left(
{\displaystyle\int_{1}^{e}}
s^{t-1}ds\right)  ^{\alpha}.$ From (\ref{2}) and (\ref{6}), we have%
\begin{align*}
&  \frac{d^{k}}{dt^{k}}\left(  \dfrac{e^{t}-1}{t}\right)  ^{-\alpha}\\
&  =%
{\displaystyle\sum\limits_{i=0}^{k}}
\left\langle -\alpha\right\rangle _{i}\left(  \dfrac{e^{t}-1}{t}\right)
^{-\alpha-i}B_{k,i}\left(
{\displaystyle\int_{1}^{e}}
s^{t-1}\ln sds,%
{\displaystyle\int_{1}^{e}}
s^{t-1}\ln^{2}sds,...,%
{\displaystyle\int_{1}^{e}}
s^{t-1}\ln^{k-i+1}sds\right) \\
&  \rightarrow%
{\displaystyle\sum\limits_{i=0}^{k}}
\left\langle -\alpha\right\rangle _{i}B_{k,i}\left(  \frac{1}{2},\frac{1}%
{3},...,\frac{1}{n-i+2}\right)  ,\text{ \ \ \ \ as }t\rightarrow0\\
&  =%
{\displaystyle\sum\limits_{i=0}^{k}}
\left\langle -\alpha\right\rangle _{i}\frac{k!}{\left(  k+i\right)  !}%
{\displaystyle\sum\limits_{j=0}^{i}}
\left(  -1\right)  ^{i-j}\binom{k+i}{i-j}S\left(  k+j,j\right)  ,
\end{align*}
Also $\left(  e^{xt}\right)  ^{(k)}=x^{k}e^{xt}\rightarrow x^{k},$ as
$t\rightarrow0.$ So, using the Leibnitz's formula for the $n$th derivative of
the product of two functions yields that
\begin{align*}
&  \lim_{t\rightarrow0}\frac{d^{n}}{dt^{n}}\left[  \left(  \dfrac{e^{t}-1}%
{t}\right)  ^{-\alpha}e^{xt}\right] \\
&  =%
{\displaystyle\sum\limits_{k=0}^{n}}
\binom{n}{k}%
{\displaystyle\sum\limits_{i=0}^{k}}
\left\langle -\alpha\right\rangle _{i}\frac{k!}{\left(  k+i\right)  !}%
{\displaystyle\sum\limits_{j=0}^{i}}
\left(  -1\right)  ^{i-j}\binom{k+i}{i-j}S\left(  k+j,j\right)  x^{n-k},
\end{align*}
which is equal to $B_{n}^{(\alpha)}(x)$ from the generating function
(\ref{0}). For $x=0,$ we immediately get the identity (\ref{11}).
\end{proof}

\begin{remark}
For $\alpha=1,$ noting the fact $\left\langle -1\right\rangle _{i}=\left(
-1\right)  ^{i}i!$, the equation (\ref{11}) becomes%
\begin{align*}
B_{n}  &  =%
{\displaystyle\sum\limits_{i=0}^{n}}
i!\frac{n!}{(n+i)!}%
{\displaystyle\sum\limits_{j=0}^{i}}
\left(  -1\right)  ^{j}\binom{n+i}{i-j}S(n+j,j)\\
&  =%
{\displaystyle\sum\limits_{j=0}^{n}}
\left(  -1\right)  ^{j}\frac{S(n+j,j)}{\binom{n+j}{j}}%
{\displaystyle\sum\limits_{i=j}^{n}}
\binom{i}{j}\\
&  =%
{\displaystyle\sum\limits_{j=0}^{n}}
\left(  -1\right)  ^{j}\frac{\binom{n+1}{j+1}}{\binom{n+j}{j}}S(n+j,j),
\end{align*}
which is \cite[Eq. 1.3]{qi2}.
\end{remark}

\begin{theorem}
\label{main1}The higher-order Euler polynomials $E_{n}^{(\alpha)}(x)$ can be
represented as
\[
E_{n}^{(\alpha)}(x)=%
{\displaystyle\sum\limits_{k=0}^{n}}
\binom{n}{k}%
{\displaystyle\sum\limits_{i=0}^{k}}
\left\langle -\alpha\right\rangle _{i}\frac{S(k,i)}{2^{i}}x^{n-k}.
\]
In particular, the higher-order Euler numbers $E_{n}^{(\alpha)}$ have the
following form%
\begin{equation}
E_{n}^{(\alpha)}=%
{\displaystyle\sum\limits_{k=0}^{n}}
\binom{n}{k}%
{\displaystyle\sum\limits_{i=0}^{k}}
\left\langle -\alpha\right\rangle _{i}S(k,i)2^{k-i}. \label{16}%
\end{equation}

\end{theorem}

\begin{proof}
By (\ref{2}), (\ref{5}) and (\ref{12}), we have%
\begin{align*}
\frac{d^{k}}{dt^{k}}\left(  \dfrac{e^{t}+1}{2}\right)  ^{-\alpha}  &  =%
{\displaystyle\sum\limits_{i=0}^{k}}
\left\langle -\alpha\right\rangle _{i}\left(  \dfrac{e^{t}+1}{2}\right)
^{-\alpha-i}B_{k,i}\left(  \frac{e^{t}}{2},\frac{e^{t}}{2},...,\frac{e^{t}}%
{2}\right) \\
&  =%
{\displaystyle\sum\limits_{i=0}^{k}}
\left\langle -\alpha\right\rangle _{i}\left(  \dfrac{e^{t}+1}{2}\right)
^{-\alpha-i}\left(  \frac{e^{t}}{2}\right)  ^{i}B_{k,i}\left(
1,1,...,1\right) \\
&  \rightarrow%
{\displaystyle\sum\limits_{i=0}^{k}}
\left\langle -\alpha\right\rangle _{i}\frac{S\left(  k,i\right)  }{2^{i}%
},\text{ \ \ \ \ \ \ as }t\rightarrow0.
\end{align*}
So, from the Leibnitz's rule again and the generating function for
higher-order Euler polynomials, given by (\ref{15}), we obtain that%
\begin{align*}
E_{n}^{(\alpha)}(x)  &  =\lim_{t\rightarrow0}\frac{d^{n}}{dt^{n}}\left[
\left(  \dfrac{e^{t}+1}{2}\right)  ^{-\alpha}e^{xt}\right] \\
&  =%
{\displaystyle\sum\limits_{k=0}^{n}}
\binom{n}{k}%
{\displaystyle\sum\limits_{i=0}^{k}}
\left\langle -\alpha\right\rangle _{i}\frac{S\left(  k,i\right)  }{2^{i}%
}x^{n-k}.
\end{align*}
Taking special case gives the closed form (\ref{16}) immediately for
higher-order Euler numbers.
\end{proof}

\begin{remark}
For $\alpha=1,$ the counterpart closed formula for Euler polynomials can be
derived. Moreover, setting more special cases leads to similar formula for
Euler numbers.
\end{remark}

\section{Determinantal expressions}

This section is devoted to demonstrate some determinantal expressions for
higher order Bernoulli and Euler polynomials.

\begin{theorem}
\label{main2}The higher-order Bernoulli polynomials $B_{n}^{(\alpha)}(x)$ can
be represented in terms of the following determinant as%
\[
B_{n}^{(\alpha)}\left(  x\right)  =\left(  -1\right)  ^{n}\left\vert
\begin{array}
[c]{cccccc}%
1 & \gamma_{0} & 0 & \ldots & 0 & 0\\
x & \gamma_{1} & \gamma_{0} & \ldots & 0 & 0\\
x^{2} & \gamma_{2} & \binom{2}{1}\gamma_{1} & \ldots & 0 & 0\\
\vdots & \vdots & \vdots & \ddots & \vdots & \vdots\\
x^{n-2} & \gamma_{n-2} & \binom{n-2}{1}\gamma_{n-3} & ... & \gamma_{0} & 0\\
x^{n-1} & \gamma_{n-1} & \binom{n-1}{1}\gamma_{n-2} & \ldots & \binom
{n-1}{n-2}\gamma_{1} & \gamma_{0}\\
x^{n} & \gamma_{n} & \binom{n}{1}\gamma_{n-1} & \ldots & \binom{n}{n-2}%
\gamma_{2} & \binom{n}{n-1}\gamma_{1}%
\end{array}
\right\vert ,
\]
where
\[
\gamma_{n}=%
{\displaystyle\sum\limits_{i=0}^{n}}
\left\langle \alpha\right\rangle _{i}\frac{n!}{\left(  n+i\right)  !}%
{\displaystyle\sum\limits_{j=0}^{i}}
\left(  -1\right)  ^{i-j}\binom{n+i}{i-j}S\left(  n+j,j\right)  .
\]

\end{theorem}

\begin{proof}
We use Lemma \ref{1} for $p\left(  t\right)  =e^{xt}$ and $q\left(  t\right)
=\left(  \left(  e^{t}-1\right)  /t\right)  ^{\alpha}.$ Note that if we
proceed similar manipulations to the proof of Theorem \ref{main}, then, we
deduce that
\[
\lim_{t\rightarrow0}\frac{d^{n}}{dt^{n}}q\left(  t\right)  =%
{\displaystyle\sum\limits_{i=0}^{n}}
\left\langle \alpha\right\rangle _{i}\frac{n!}{\left(  n+i\right)  !}%
{\displaystyle\sum\limits_{j=0}^{i}}
\left(  -1\right)  ^{i-j}\binom{n+i}{i-j}S\left(  n+j,j\right)  :=\gamma_{n}.
\]
So, we have%
\begin{align*}
&  \frac{d^{n}}{dt^{n}}\left[  \frac{e^{xt}}{\left(  \left(  e^{t}-1\right)
/t\right)  ^{\alpha}}\right] \\
&  =\frac{\left(  -1\right)  ^{n}}{\left(  \left(  e^{t}-1\right)  /t\right)
^{(n+1)\alpha}}\left\vert
\begin{array}
[c]{cccccc}%
e^{xt} & q & 0 & \ldots & 0 & 0\\
xe^{xt} & q^{\prime} & q & \ldots & 0 & 0\\
x^{2}e^{xt} & q^{\prime\prime} & \binom{2}{1}q^{\prime} & \ldots & 0 & 0\\
\vdots & \vdots & \vdots & \ddots & \vdots & \vdots\\
x^{n-2}e^{xt} & q^{(n-2)} & \binom{n-2}{1}q^{(n-3)} & ... & q & 0\\
x^{n-1}e^{xt} & q^{(n-1)} & \binom{n-1}{1}q^{(n-2)} & \ldots & \binom
{n-1}{n-2}q^{\prime} & q\\
x^{n}e^{xt} & q^{(n)} & \binom{n}{1}q^{(n-1)} & \ldots & \binom{n}%
{n-2}q^{\prime\prime} & \binom{n}{n-1}q^{\prime}%
\end{array}
\right\vert \\
&  \rightarrow\left(  -1\right)  ^{n}\left\vert
\begin{array}
[c]{cccccc}%
1 & \gamma_{0} & 0 & \ldots & 0 & 0\\
x & \gamma_{1} & \gamma_{0} & \ldots & 0 & 0\\
x^{2} & \gamma_{2} & \binom{2}{1}\gamma_{1} & \ldots & 0 & 0\\
\vdots & \vdots & \vdots & \ddots & \vdots & \vdots\\
x^{n-2} & \gamma_{n-2} & \binom{n-2}{1}\gamma_{n-3} & ... & \gamma_{0} & 0\\
x^{n-1} & \gamma_{n-1} & \binom{n-1}{1}\gamma_{n-2} & \ldots & \binom
{n-1}{n-2}\gamma_{1} & \gamma_{0}\\
x^{n} & \gamma_{n} & \binom{n}{1}\gamma_{n-1} & \ldots & \binom{n}{n-2}%
\gamma_{2} & \binom{n}{n-1}\gamma_{1}%
\end{array}
\right\vert
\end{align*}
as $t\rightarrow0.$ From the generating function, given by (\ref{0}), we reach
the desired result.
\end{proof}

\begin{remark}
We mention that for the special case $x=0,$ the analog determinantal
expression for higher-order Bernoulli numbers $B_{n}^{(\alpha)}$ can be
offered. Moreover, for $\alpha=1,$ and $\alpha=1$ and $x=0,$ the similar
representations can be obtained for classical Bernoulli polynomials and
numbers, respectively.
\end{remark}

\begin{theorem}
The higher-order Euler polynomials $E_{n}^{(\alpha)}(x)$ can be represented in
terms of the following determinant as%
\[
E_{n}^{(\alpha)}\left(  x\right)  =\left(  -1\right)  ^{n}\left\vert
\begin{array}
[c]{cccccc}%
1 & \beta_{0} & 0 & \ldots & 0 & 0\\
x & \beta_{1} & \beta_{0} & \ldots & 0 & 0\\
x^{2} & \beta_{2} & \binom{2}{1}\beta_{1} & \ldots & 0 & 0\\
\vdots & \vdots & \vdots & \ddots & \vdots & \vdots\\
x^{n-2} & \beta_{n-2} & \binom{n-2}{1}\beta_{n-3} & ... & \beta_{0} & 0\\
x^{n-1} & \beta_{n-1} & \binom{n-1}{1}\beta_{n-2} & \ldots & \binom{n-1}%
{n-2}\beta_{1} & \beta_{0}\\
x^{n} & \beta_{n} & \binom{n}{1}\beta_{n-1} & \ldots & \binom{n}{n-2}\beta_{2}
& \binom{n}{n-1}\beta_{1}%
\end{array}
\right\vert ,
\]
where%
\[
\beta_{n}=%
{\displaystyle\sum\limits_{i=0}^{n}}
\left\langle \alpha\right\rangle _{i}\frac{S\left(  n,i\right)  }{2^{i}}.
\]

\end{theorem}

\begin{proof}
The proof can be verified by proceeding as in the proof of Theorem
\ref{main2}. So, we omit it.
\end{proof}

\begin{remark}
The special values of the Bell polynomials of the second kind $B_{n,k}$ are
worthy in combinatorics and number theory. In this respect, $B_{n,k}$ has been
applied in order to cope with some difficult problems and obtain significant
results in many studies (see for example \cite{qi9,qi12,qi13,qi15,qi16}).
\end{remark}


\begin{thebibliography}{99}                                                                                               %


\bibitem {bourbaki}N. Bourbaki,\textit{ Functions of a Real Variable,
Elementary Theory,} Translated from the 1976 French Original by Philip Spain.
Elements of Mathematics (Berlin). Springer, Berlin, 2004.

\bibitem {closed}\textit{Closed-form expression.} https://en.wikipedia.org/wiki/Closed-form\_expression.

\bibitem {Comtet}L. Comtet, \textit{Advanced Combinatorics:} The Art of Finite
and Infinite Expansions, Revised and Enlarged Edition. D. Reidel Publishing
Co., Dordrecht 1974.

\bibitem {d}L. Dai and H. Pan, \textit{Closed forms for degenerate Bernoulli
polynomials,} Bull. Aust. Math. Soc. \textbf{101} (2020), no. 2, 207--217.

\bibitem {guo}B.-N. Guo and F. Qi, \textit{An explicit formula for Bernoulli
numbers in terms of Stirling numbers of the second kind,} J. Anal. Number
Theory \textbf{3} (2015), no. 1, 27--30.

\bibitem {hu-kim}S. Hu and M.-S. Kim, \textit{Two closed forms for the
Apostol--Bernoulli polynomials,} Ramanujan J. \textbf{46} (2018), no. 1, 103--117.

\bibitem {n}N. E. N\"{o}rlund, \textit{Vorlesungen \"{u}ber
Differenzenrechnung,} Springer-Verlag, Berlin, 1924.

\bibitem {qi2}F. Qi and R. J. Chapman, \textit{Two closed forms for the
Bernoulli polynomials,} J. Number Theory, \textbf{159} (2016), 89--100.

\bibitem {qi6}F. Qi and B.-N. Guo, \textit{Some Determinantal Expressions and
Recurrence Relations of the Bernoulli Polynomials,} Mathematics, \textbf{4}
(2016), no. 4, 1--11.

\bibitem {qi9}F. Qi, \textit{Derivatives of tangent function and tangent
numbers,} Appl. Math. Comput. \textbf{268} (2015), 844--858.

\bibitem {qi12}F. Qi and M.-M. Zheng, \textit{Explicit expressions for a
family of the Bell polynomials and applications,} Appl. Math. Comput.
\textbf{258} (2015), 597--607.

\bibitem {qi13}F. Qi and B.-N. Guo, \textit{Explicit formulas for special
values of the Bell polynomials of the second kind and for the Euler numbers
and polynomials,} Mediterr. J.\ Math. \textbf{14} (2017), no. 3, Article 140,
14 pages.

\bibitem {qi15}F. Qi, D. Lim, and B.-N. Guo, \textit{Explicit formulas and
identities for the Bell polynomials and a sequence of polynomials applied to
differential equations,} Rev. R. Acad. Cienc. Exactas F\'{\i}s. Nat. Ser. A
Mat. RACSAM, \textbf{113} (2019), 1--9.

\bibitem {qi16}F. Qi, \textit{An Explicit Formula for the Bell Numbers in
Terms of the Lah and Stirling Numbers,} Mediterr. J. Math. \textbf{13} (2016), 2795--2800.

\bibitem {wei}C.-F. Wei and F. Qi, \textit{Several closed expressions for the
Euler numbers,} J. Inequal. Appl. 2015, \textbf{219} (2015).
\end{thebibliography}
\end{document}